\documentclass[en,oneside,bbfont,article]{amsart} 
\usepackage[lmargin = 30mm, rmargin = 30mm, bmargin = 25mm, tmargin=25mm]{geometry}
\usepackage[utf8]{inputenc}

\usepackage[russian,english]{babel}
\usepackage[OT2, T1]{fontenc}

\usepackage{mathtools}

\providecommand{\examplename}{Example}

\newtheorem{theorem}{Theorem}[section]
\newtheorem{proposition}[theorem]{Proposition}
\newtheorem{corollary}[theorem]{Corollary}
\newtheorem{lemma}[theorem]{Lemma}
\theoremstyle{remark}
\newtheorem{remarkx}[theorem]{Remark}
\newenvironment{remark}
    {\pushQED{\qed}\remarkx}
    {\popQED\endremarkx}
\theoremstyle{definition}
\newtheorem*{example*}{\protect\examplename}
\newtheorem{examplex}[theorem]{\protect\examplename}
\newenvironment{example}
    {\pushQED{\qed}\examplex}
    {\popQED\endexamplex}
\theoremstyle{definition}

\newtheorem*{assumption*}{Assumption}

\usepackage{mathtools}
\usepackage{mathrsfs}
\usepackage{pbox}
\usepackage{amssymb}
\usepackage{amstext}
\usepackage{dsfont}
\usepackage{amsthm}
\usepackage{multicol}
\usepackage{amsmath}
\usepackage{bm}
\usepackage{comment}
\usepackage{caption,subcaption,enumerate,enumitem}
\usepackage{scalefnt}
\usepackage{float,graphicx,verbatim}
\usepackage{tikz}
\usetikzlibrary{calc,shapes,arrows,positioning,decorations.pathreplacing,calligraphy}
\usepackage{pgfplots}
\pgfplotsset{compat=1.16}
\RequirePackage[colorlinks=true,
	citecolor=blue,urlcolor=blue]{hyperref}

\usepackage{multirow}

\usepackage[backend=biber,style=numeric,citestyle=numeric,maxbibnames=99,sortcites=true]{biblatex}
\renewbibmacro{in:}{}
\appto{\bibsetup}{\sloppy}

\newcommand\N{\mathbb{N}}

\newcommand{\mS}{\mathcal{S}}

\newcommand\R{\mathbb{R}}
\newcommand\E{\mathds{E}}
\newcommand\p{\mathds{P}}
\newcommand\1{\mathds{1}}

\newcommand\da{\downarrow}
\newcommand\ua{\uparrow}

\newcommand{\ve}{\varepsilon}
\newcommand{\cadlag}{{c\`adl\`ag}}

\newcommand\eqd{\overset{d}{=}}

\newcommand\nf[1]{\normalfont{#1}}

\newcommand{\D}{\mathrm{d}}
\newcommand{\ov}[1]{\overline{#1}}

\usepackage[foot]{amsaddr}

\title{H\"older continuity of the convex minorant of a L\'evy process}
\author{Jorge Gonz\'alez C\'azares$^{*\dag}$, David Kramer-Bang$^*$ \& Aleksandar Mijatovi\'c$^{*\dag}$}

\address{$^*$Department of Statistics, University of Warwick, \and $^\dag$The Alan Turing Institute, UK}

\email{jorge.ignacio.gc@gmail.com}
\email{davidbang.privat@gmail.com}
\email{a.mijatovic@warwick.ac.uk}

\begin{document}

\begin{abstract}
We characterise the H\"older continuity of the convex minorant of most L\'evy processes. The proof is based on a novel connection between the path properties of the L\'evy process at zero and the boundedness of the set of $r$-slopes of the convex minorant.
\end{abstract}



\maketitle

\section{Introduction and main results }
\label{sec:intro}



The H\"{o}lder continuity\footnote{Given $r \in (0,1]$ the function  $f:[0,T] \to \R$ is $r$-H\"{o}lder continuous by definition if its H\"older constant $\sup_{0\le x<y\le T}|f(y)-f(x)|/(y-x)^r$ is finite.} of continuous random functions is a classical topic, analysed extensively for Brownian motion and related processes, see e.g.~\cite{MR3642847} for fractional Brownian motion. Typically, such results make use of Kolmogorov's extension theorem (see~\cite[Thm~3.23]{MR1876169}). The convex minorant of a L\'evy process is a  continuous random function, which may but need not be smooth~\cite{SmoothCM,HowSmoothCM}, motivating the question of its H\"{o}lder continuity. However, as the increments of the convex minorant (and their moments) are not tractable and its local behaviour varies greatly with the characteristics of the L\'evy process~\cite{SmoothCM}, Kolmogorov's extension theorem is not the right tool. In this paper we establish sufficient and necessary conditions for the H\"{o}lder continuity of the convex minorant of a L\'evy process, using a generalisation of the 0--1 law in~\cite[Thm~3.1]{SmoothCM}, the characterisation of small time behaviour of the L\'evy path in~\cite[Thm~2.1]{MR2370602} and an elementary lemma by Khinchine (see Lemma~\ref{lem:Khintchine} below). We prove for example that, in the absence of a Brownian component, the critical H\"{o}lder exponent is the reciprocal of the Blumenthal--Getoor index for most infinite variation L\'evy processes (complete results are given in Table~\ref{tab:Holder} below). A short \href{https://youtu.be/PKvSg2tKqfs}{YouTube}~\cite{Presentation_Hoelder} video describes the results and the structure of our proofs. 


The convex minorant $C$ of the path of the L\'evy process $X$ on $[0,T]$ is the pointwise largest convex function dominated by $X$.   By~\cite[Prop.~1.3]{SmoothCM}, $C$ is Lipschitz (i.e., $1$-H\"older) continuous if and only if $X$ is of finite variation. In what follows we assume that $X$ is of infinite variation. Then $C$ is not Lipschitz on $[0,T]$ but, by convexity, is Lipschitz on every interval $[\ve,T-\ve]$, $\ve>0$, with Lipschitz constant $\max\{|C'_\ve|,C'_{T-\ve}\}$ given in terms of the right-derivative $C'$ of $C$. Note that the rate at which $\max\{|C'_\ve|,C'_{T-\ve}\}$ tends to infinity as $\ve\da 0$, analysed in~\cite{HowSmoothCM}, is insufficient to characterise the $r$-H\"older continuity of $C$ on $[0,T]$ for $r\in(0,1)$, since $C'_\ve$ may fluctuate between functions that are not asymptotically equivalent as $\ve\da 0$, see~\cite[Rem.~2.14(a)]{HowSmoothCM} (analogous behaviour is observed for $C'_{T-\ve}$). 

Let $\sigma$ and $\nu$ be the Gaussian coefficient and L\'evy measure of $X$, respectively (see~\cite{MR3185174} for background on L\'evy processes). Let $\beta$ denote the Blumenthal--Getoor index of $X$ given by
\begin{equation}\label{eq:BG}
\beta
\coloneqq\inf\big\{p>0\,:\,J_{p}<\infty\big\}, 
\quad \text{where} \quad 
J_p\coloneqq \int_{(-1,1)}|x|^{p}\nu(\D x).
\end{equation} 
Note that, when $\sigma=0$, $\beta\in[1,2]$ since $X$ is of infinite variation. (If $\sigma>0$, we may have $\beta\in[0,1)$.) Since $J_2<\infty$, we have  $\ov\sigma^2(u)\coloneqq\int_{(-u,u)}x^2\nu(\D x)<\infty$ for $u>0$ and $\ov\sigma^2(u)\to0$ as $u\to0$ (but the convergence may be arbitrarily slow). Define also
\begin{equation}
\label{eq:lambda}
\lambda_2\coloneqq\inf\bigg\{\lambda>0:\int_0^1e^{-\lambda^2/(2\ov\sigma^2(u))}\frac{\D u}{u}<\infty\bigg\}
\quad\&\quad
I_2
\coloneqq\int_0^1\E[\min\{(X_t/\sqrt{t})^2,1\}]\frac{\D t}{t}.
\end{equation}

Theorem~\ref{thm:main_theorem}, summarised in Table~\ref{tab:Holder} below, characterises, for most L\'evy process of infinite variation, the set of $r\in(0,1)$ for which $C$ is $r$-H\"older continuous. The characterisation is given in terms of $J_\beta$, except for $\beta=2$, where $\lambda_2$ and $I_2$ are required.

\begin{table}[ht]
\begin{center}
\begin{tabular}{|c|c|c|c|c|}
\hline
\multicolumn{3}{|c|}{Infinite variation L\'evy process $X$} 
& $r\in(0,1)$ 
& Is $C$ $r$-H\"older continuous?\\
\hline
\hline
\multicolumn{3}{|c|}{\multirow{2}{*}{$\sigma^2>0$}}
& $r<1/2$
& Yes\\
\cline{4-5}
\multicolumn{3}{|c|}{}
& $r\ge 1/2$
& No\\
\hline
\multirow{12}{*}{$\sigma^2=0$}
& \multicolumn{2}{|c|}{$\beta=1$}
& $r\in(0,1)$
& Yes\\
\cline{2-5}
& \multirow{4}{*}{$\beta\in(1,2)$}
& \multirow{2}{*}{$J_{\beta}=\infty$}
& $r<1/\beta$ 
& Yes\\
\cline{4-5}
& & & $r\ge 1/\beta$ 
& No\\
\cline{3-5}
& 
& \multirow{2}{*}{$J_{\beta}<\infty$}
& $r\le1/\beta$
& Yes\\
\cline{4-5}
& & & $r>1/\beta$ 
& No\\
\cline{2-5}
& \multirow{7}{*}{$\beta=2$}
& \multirow{2}{*}{$\lambda_2=\infty$}
& $r<1/2$ 
& Yes\\
\cline{4-5}
& & & $r\ge1/2$ 
& No\\
\cline{3-5}
& 
& \multirow{2}{*}{$I_2<\infty$}
& $r\le 1/2$
& Yes\\
\cline{4-5}
& & & $r>1/2$ 
& No\\
\cline{3-5}
& & \multirow{3}{*}{\pbox{4.25cm}{\centering $\lambda_2<\infty$ and
$I_2=\infty$}}
& $r<1/2$ 
& Yes\\
\cline{4-5}
& & & $r>1/2$ 
& No\\
\cline{4-5}
& & & $r=1/2$ 
& Inconclusive\\
\hline
\end{tabular}
\end{center}
\caption{\label{tab:Holder}
The critical level $r\in(0,1)$ for $r$-H\"older continuity of the convex minorant is $1/2$ in the presence of a Brownian component and $1/\beta$ in its absence, where $\beta\in[1,2]$ is the Blumenthal--Getoor index.}
\end{table}

\subsection{\texorpdfstring{$r$}{r}-H\"older continuity and sets of \texorpdfstring{$r$}{r}-slopes}
\label{subsec:main}

The convex minorant $C$ is piecewise linear with countably many maximal intervals of linearity (see, e.g.~\cite{fluctuation_levy}). Denote the corresponding sequences of horizontal lengths and vertical heights by $(\ell_n)_{n\in\N}$ and $(\xi_n)_{n\in\N}$, respectively. Thus, over the $n$-th interval of linearity (where $C$ has slope $\xi_n/\ell_n$), $C$ is clearly $r$-H\"older with H\"older constant $|\xi_n|/\ell_n^r$. Our main objective is to characterise when a L\'evy processes has a convex minorant that is $r$-H\"older continuous for $r\in(0,1)$. It turns out that, for a large class of L\'evy processes, the a.s. finiteness of $k_r\coloneqq\sup_{n\in\N}|\xi_n|/\ell_n^r$ implies that $C$ is $r$-H\"older a.s. It is important to note that neither $0$ nor $T$ are the endpoints of an interval of linearity of $C$ since $X$ is of infinite variation~\cite[Sec.~1.1.2]{SmoothCM}, implying that, even though $C$ is always ``locally $r$-H\"older'' on $(0,T)$ by convexity (i.e. $r$-H\"older on every compact subinterval of $(0,T)$), it may fail to be $r$-H\"older on $[0,T]$.

For any $r\in(0,1)$, define the set of $r$-slopes by $\mS_r\coloneqq\{\xi_n/\ell_n^r:n\in\N\}$, which is either a.s. bounded ($k_r<\infty$) or a.s. unbounded ($k_r=\infty$) by~\cite[Cor.~3.2]{SmoothCM}. 
By Lemma~\ref{lem:PC-Holder} below, we have:
\begin{equation}
\label{eq:Holder-K_r}
k_r
=\sup_{s\in\mS_r}|s|
\le\sup_{0\le u<t\le T}\frac{|C_t-C_u|}{(t-u)^r}
\le\bigg(\sum_{s \in \mS_r} |s|^{1/(1-r)}\bigg)^{1-r}
\eqqcolon K_r
\quad\text{a.s.}
\end{equation}
Note that the upper bound $K_r$ on the $r$-H\"older constant in~\eqref{eq:Holder-K_r} is in fact the $L^{p}$-norm of $C'$ for $p=1/(1-r)$. Furthermore, the first inequality in~\eqref{eq:Holder-K_r} may be strict, see Example~\ref{ex:Holder-K_r-strict} below.

The utility of~\eqref{eq:Holder-K_r} lies in the fact that it controls the H\"older continuity of convex minorant $C$, since $C$ is $r$-H\"older if $K_r<\infty$ and it is \emph{not} $r$-H\"older if $k_r=\infty$. Our main result, Theorem~\ref{thm:main_theorem}, shows that, for all L\'evy processes and $r\in (0,1)\setminus\{1/2\}$ (and even $r=1/2$ if either $\sigma^2>0$ or $\beta\in[1,2)$), $k_r$ and $K_r$ are simultaneously finite or infinite: $\p(\{K_r=\infty\}\cap\{k_r<\infty\})=0$, yielding Table~\ref{tab:Holder}. Since, by Proposition~\ref{prop:SB-sum} below, for any L\'evy process $X$ and any $r\in(0,1)$, we have $\p(k_r=\infty)\in\{0,1\}$ and $\p(K_r=\infty)\in\{0,1\}$, the main function of Theorem~\ref{thm:main_theorem} is thus to rule out the possibility of having $k_r<\infty=K_r$ a.s. Only in the extreme case $\sigma^2=0$, $\beta=2$ and $r=1/2$, does there exist a L\'evy processes for which $k_r<\infty=K_r$ a.s., making~\eqref{eq:Holder-K_r} inconclusive in determining if $C$ is $\tfrac{1}{2}$-H\"older, see Proposition~\ref{prop:b=2} and Example~\ref{ex:b=2} below.

\begin{theorem}\label{thm:main_theorem}
Let $X$ be a L\'evy process of infinite variation and $r\in(0,1)$.
\begin{itemize}
\item[{\nf{(i)}}] If $\sigma^2>0$, then $k_{r}=\infty$ a.s. for $r\in[1/2,1)$ and $K_r<\infty$ a.s. for $r\in(0,1/2)$. 
\item[{\nf{(ii)}}] If $\sigma^2=0$ and $\beta\in[1,2)$, then the following equivalences hold: 
$$k_{r}<\infty\text{ a.s. }\iff J_{1/r}<\infty\iff K_{r}<\infty\text{ a.s.}$$
\item[{\nf{(iii)}}] If $\sigma^2=0$ and $\beta=2$, then $k_{r}=\infty$ a.s. for $r\in(1/2,1)$ and $K_r<\infty$ a.s. for $r\in(0,1/2)$.
If $r=1/2$, the following implications hold:\\ 
\nf{(a)} $\lambda_2=\infty\implies k_{1/2}=\infty$ a.s., \qquad\,\,\nf{(b)}
$\lambda_2\in(0,\infty)\implies k_{1/2}<\infty=K_{1/2}$ a.s.,\\  \nf{(c)} $\lambda_2=0\implies k_{1/2}<\infty$ a.s.,  \qquad\quad\nf{(d)} $I_2<\infty\implies K_{1/2}<\infty$ a.s.
\end{itemize}
\end{theorem}

We note that the \emph{only} inconclusive cases in Theorem~\ref{thm:main_theorem} are (iii-b) and (iii-c). In fact, Example~\ref{ex:b=2} below presents a L\'evy process as in Theorem~\ref{thm:main_theorem}(iii-b). The proof of Theorem~\ref{thm:main_theorem} rests on Propositions~\ref{prop:not-1/b}--\ref{prop:b=2}, which we introduce next.

\begin{proposition}
\label{prop:not-1/b}
Let $X$ be a L\'evy process of infinite variation. If $\sigma^2>0$, then $k_{r}=\infty$ a.s. for $r\in[1/2,1)$ and $K_r<\infty$ a.s. for $r\in(0,1/2)$. If $\sigma^2=0$, then $k_{r}=\infty$ a.s. for $r\in(1/\beta,1)$ and $K_r<\infty$ a.s. for $r\in(0,1/\beta)$.
\end{proposition}

By the inequalities in~\eqref{eq:Holder-K_r}, Proposition~\ref{prop:not-1/b} characterises H\"older continuity of $C$ when either $\beta=1$ or $\sigma^2>0$, implying the rows one, two and three in Table~\ref{tab:Holder}. Moreover, Proposition~\ref{prop:not-1/b} reveals that the critical level of the H\"older exponent is $r=1/\beta$ with $\beta\in(1,2]$, considered next. Define
\begin{equation}\label{eq:I_b}
I_\beta\coloneqq\int_0^1\E\big[\min\{(|X_t|/t^{1/\beta})^{\beta/(\beta-1)},1\}\big]\frac{\D t}{t}\in(0,\infty],
\qquad
\beta\in(1,2].
\end{equation}

\begin{proposition}
\label{prop:1/b}
Let $X$ be a L\'evy process of infinite variation with $\sigma^2=0$ and $\beta\in(1,2]$. Then the following   holds: {\nf(i)} $J_\beta=\infty\!\iff\! k_{1/\beta}=\infty$ a.s.; {\nf(ii)} $I_\beta<\infty\!\iff\! K_{1/\beta}<\infty$ a.s.
\end{proposition}

Implicit in Proposition~\ref{prop:1/b} is the fact that $I_\beta<\infty$ implies $J_\beta<\infty$. Checking the finiteness of the integral $I_\beta$ in  Proposition~\ref{prop:1/b}(ii) may appear hard as it is given in~\eqref{eq:I_b} in terms of the truncated moments of the marginals of $X$. We now give sufficient conditions for $I_\beta<\infty$ in terms of the L\'evy measure $\nu$, implying, in particular, that $J_\beta<\infty\iff I_\beta<\infty$ when $\beta\in(1,2)$. Define 
\begin{equation}
\label{eq:ov_nu}
\ov\nu(u)\coloneqq \nu(\R\setminus(-u,u))
\quad\text{for $u\in(0,1]$.}
\end{equation}
Note that, by Fubini's theorem, $J_p=\int_{(-1,1)}|x|^p\nu(\D x)=\int_0^1\ov\nu(t^{1/p})\D t-\ov\nu(1)$ for any $p>0$. In particular, the condition $J_\beta<\infty$ is equivalent to $\int_0^1\ov\nu(t^{1/\beta})\D t<\infty$.

\begin{proposition}
\label{prop:suff_1/b}
Let $X$ be an infinite variation L\'evy process with $\sigma^2=0$. If $\beta\in(1,2)$, then $J_\beta<\infty\iff I_\beta<\infty$. If $\beta=2$, then $\int_0^1\log(1/t)\ov\nu(\sqrt{t})\D t<\infty$ implies $I_2<\infty$.
\end{proposition}

By Proposition~\ref{prop:suff_1/b}, for any process with $\beta\in(1,2)$, Proposition~\ref{prop:1/b} characterises the $(1/\beta)$-H\"older continuity of $C$. Moreover, since $\int_0^1\ov\nu(\sqrt{t})\D t<\infty$ for any L\'evy measure $\nu$, Proposition~\ref{prop:suff_1/b} shows that $I_2<\infty$ for many L\'evy measures.

If $X$ has no Brownian component (i.e. $\sigma^2=0$) but satisfies $\beta=2$, it is possible to have $k_{1/2}<\infty$ and $K_{1/2}=\infty\text{ a.s.}$, rendering~\eqref{eq:Holder-K_r}  insufficient to ascertain whether $C$ is $\tfrac{1}{2}$-H\"older continuous. Indeed, the phenomenon $k_{1/2}<\infty=K_{1/2}$ occurs whenever the constant $\lambda_2$ in~\eqref{eq:lambda} lies in $(0,\infty)$. In fact, we have the following. 

\begin{proposition}
\label{prop:b=2}
Suppose that $\sigma^2=0$ and $\beta=2$. Then \nf{(i)} $\lambda_2=\infty$ implies $k_{1/2}=\infty$ a.s., \nf{(ii)}
$\lambda_2\in(0,\infty)$ implies $k_{1/2}<\infty=K_{1/2}$ a.s., and \nf{(iii)} $\lambda_2=0$ implies $k_{1/2}<\infty$ a.s.
\end{proposition}

Observe that $\lambda_2=0$ whenever $\Lambda(e^{-e})<\infty$ where $\Lambda(x)\coloneqq\int_0^x(\log\log(1/t))\ov\nu(\sqrt{t})\D t$ (see~\cite[Rem.~2(iii)]{MR2370602}) since, by Fubini's theorem, $\ov\sigma^2(x)\le\int_0^x\ov\nu(\sqrt{t})\D t\le \Lambda(x)/\log\log(1/x)$ for $x\in(0,e^{-e})$ and $\lim_{x\da 0}\Lambda(x)=0$. Note that the condition $\int_0^1\log(1/t)\ov\nu(\sqrt{t})\D t<\infty$ in Proposition~\ref{prop:suff_1/b} is stronger than $\Lambda(e^{-e})<\infty$, so that, even in the case where $\Lambda(e^{-e})<\infty$, we cannot establish $I_2<\infty$ via Proposition~\ref{prop:suff_1/b}. We suspect, however, that $I_2<\infty$ whenever $\lambda_2=0$. 

We conclude this subsection with two examples.

\begin{example}[{\cite[Rem~2(iii)]{MR2370602}}]
\label{ex:b=2}
Let $\nu(\D x)=\tfrac{1}{2}x^{-3}\log(1/x)^{-1}(\log\log(1/x))^{-2}\1_{(0,e^{-e})}(x)\D x$, so that $\ov\sigma^2(x)=\tfrac{1}{2}(\log\log(1/x))^{-1}\1_{(0,e^{-e})}(x)$. A direct calculation gives $\lambda_2=1\in(0,\infty)$. In this case, we clearly have $k_{1/2}<\infty=K_{1/2}$ a.s. by Proposition~\ref{prop:b=2}, making~\eqref{eq:Holder-K_r} insufficient to determine whether $C$ is $\tfrac{1}{2}$-H\"older.
\end{example}

\begin{example}
\label{ex:Holder-K_r-strict}
In this example we construct a piecewise linear convex function for which the first inequality in~\eqref{eq:Holder-K_r} is strict. Fix $r\in(0,1)$, $T>0$ and consider the piecewise linear convex function $C:t\mapsto -T^r\min\{t,T/2\}-rT^r\max\{t-T/2,0\}$ on $[0,T]$, which has two faces, both of length $T/2$ and with heights $-T^r$ and $-rT^r$, respectively. Then 
\[
\sup_{0\le u<t\le T}\frac{|C_t-C_u|}{(t-u)^r}
=\frac{|C_T-C_0|}{T^r}
=1+r
>2^r
=k_r,
\]
where the inequality follows since $g:x\mapsto 2^x-1-x$ is convex and $g(0)=g(1)=0$.
\end{example}

\subsection{Strategy for the proofs and connections with the literature}

For any $r\in(0,1)$, we give sufficient as well as necessary conditions for the convex minorant $C$ to be $r$-H\"{o}lder continuous  in terms of the set of $r$-slopes $\mS_r\coloneqq\{\xi_n/\ell_n^r:n\in\N\}$ (see~\eqref{eq:Holder-K_r} and Lemma~\ref{lem:PC-Holder}). In Proposition~\ref{prop:SB-sum} we generalise the 0--1 law of~\cite[Thm~3.1]{SmoothCM} and use it to characterise the finiteness of $K_r$ in terms of the truncated moments of the marginals of $X$. Through Khintchine's characterisation of the two-sided upper functions of $X$~\cite{MR0002054}, given in Lemma~\ref{lem:Khintchine} below, and the 0--1 law in~\cite[Cor~3.2]{SmoothCM}, we find that $k_r<\infty$ a.s. if and only if the a.s. constant (by Blumenthal's 0--1 law~\cite[Prop.~40.4]{MR1113220}) $\lambda_{1/r}=\limsup_{t\da0}|X_t|/t^r$ is finite, see Corollary~\ref{cor:Khintchine} below. The final ingredient in the proofs of Propositions~\ref{prop:not-1/b} and~\ref{prop:1/b} are the characterisations of the limit $\lambda_{1/r}$ given in~\cite[Sec.~47]{MR3185174} and~\cite{MR2370602}, respectively. In Section~\ref{sec:conclusion} we discuss a possible extension of Proposition~\ref{prop:not-1/b} and its connection to the characterisation in~\cite{MR968135} of the limits $\limsup_{t\da0}|X_t|/h(t)$ for a non-decreasing $h$.


\section{Proofs}
\label{sec:proofs}

We begin with an elementary deterministic lemma that yields the inequalities in~\eqref{eq:Holder-K_r}.

\begin{lemma}\label{lem:PC-Holder}
Let $f$ be an absolutely continuous, piecewise linear function with infinitely many faces, defined on the interval $[a,b]$. Given any enumeration of the maximal intervals of linearity of $f$, let $(l_n)_{n\in\N}$ and $(h_n)_{n\in\N}$ be the corresponding sequences of horizontal lengths and vertical heights, respectively, of those line segments. Then for any $r\in(0,1)$ we have
\[
\sup_{n \in \N} |h_n|l_n^{-r}
\le\sup_{a\le u<t\le b}\frac{|f(t)-f(u)|}{(t-u)^r}
\le\bigg(\sum_{n \in \N} (|h_n|l_n^{-r})^{1/(1-r)}\bigg)^{1-r}.
\]
\end{lemma}

\begin{proof}[Proof]
Fix $r \in (0,1)$ and let $p=1/(1-r)>1$. Let $(g_n,d_n)$, $n\in\N$, be the maximal intervals of linearity of $f$ where the slope of $f$ over $(g_n,d_n)$ equals $h_n/l_n$ and $\sum_{n\in\N}(d_n-g_n)=b-a$. The lower bound is obvious since it is attained by restricting the supremum to the values $(u,t)=(g_n,d_n)$. To establish the upper bound, first note that $f'$ exists on the set $\bigcup_{n\in\N}(g_n,d_n)$ of measure $b-a$ and 
\[
    \int_a^b|f'(t)|^p \D t= \sum_{n \in \N}\int_{g_n}^{d_n}|f'(t)|^p \D t= \sum_{n \in \N}\frac{|h_n|^p}{l_n^p} \int_{g_n}^{d_n}\D t
    =\sum_{n \in \N} \frac{|h_n|^p}{l_n^{p-1}}.
\]
By H\"older's inequality with $p$ and $q= p/(p-1)=1/r>1$, it follows 
\[
|f(t)-f(u)| 
\le\int_u^t|f'(x)|\D x
=\int_a^b\1_{[u,t]}(x)|f'(x)|\D x
\le (t-u)^{1/q} \bigg( \int_a^b |f'(x)|^p \D x\bigg)^{1/p}.
\] 
Thus, we have
\begin{align*}
\sup_{a\le u<t\le b}
    \frac{|f(t)-f(u)|}{(t-u)^{r}} &
\le\bigg( \int_a^b |f'(x)|^p \D x\bigg)^{1/p} \\
& =\bigg(\sum_{n \in \N} \frac{|h_n|^p}{l_n^{p-1}}\bigg)^{1/p}
=\bigg(\sum_{n \in \N} (|h_n|l_n^{-r})^{1/(1-r)}\bigg)^{1-r}.
\qedhere
\end{align*}
\end{proof}

The proofs of Propositions~\ref{prop:not-1/b} \& \ref{prop:1/b} hinge on two key tools. First is the 0--1 law in Proposition~\ref{prop:SB-sum}, generalising~\cite[Thm~3.1]{SmoothCM} to unbounded functionals of the faces of $C$, and second is Khintchine's characterisation of the upper functions of $|X|$ at zero given in Lemma~\ref{lem:Khintchine} below. 
Recall that since $X$ is of infinite activity, its convex minorant $C$ is a piecewise linear function whose \emph{maximal} intervals of linearity have corresponding sequences of horizontal lengths $(\ell_n)_{n\in\N}$ and vertical heights $(\xi_n)_{n\in\N}$ given by the formulae in~\cite[Thm~3.1]{fluctuation_levy}. (If $X$ is of finite activity, the intervals of linearity of the stick-breaking representation in~\cite[Thm~3.1]{fluctuation_levy} are \emph{not maximal} and, in fact, all but finitely many faces have the same slope).

\begin{proposition}
\label{prop:SB-sum}
Let $\phi:\R\times(0,\infty)\to[0,\infty)$ be measurable. Then the sum $\sum_{n\in\N}\phi(\xi_n,\ell_n)$ is either a.s. finite or a.s. infinite. Moreover, we have
\begin{align}
\label{eq:SB-sum}
\sum_{n\in\N}\phi(\xi_n,\ell_n)<\infty\quad\text{a.s.}
\quad&\iff\quad
\int_0^1\E[\min\{\phi(X_t,t), 1\}]\frac{\D t}{t}<\infty.
\end{align}
\end{proposition}

\begin{proof}[Proof]
Note that $\sum_{n\in\N}a_n<\infty$ if and only if $\sum_{n\in\N}\min\{a_n,1\}<\infty$ for any sequence $(a_n)_{n\in\N}$ in $[0,\infty)$. Thus, 
$\sum_{n\in\N}\phi(\xi_n,\ell_n)<\infty
\text{ a.s.}\iff\sum_{n\in\N}\min\{\phi(\xi_n,\ell_n),1\}<\infty\text{ a.s.}$ and the equivalence in~\eqref{eq:SB-sum} 
follows from~\cite[Thm~3.1]{fluctuation_levy} and 
the 0--1 law~\cite[Thm~3.1]{SmoothCM} applied to the bounded function $(t,x)\mapsto \min\{\phi(x,t),1\}$.
\end{proof}

The following characterisation due to Khintchine~\cite{MR0002054} is central in relating the upper fluctuations of $|X|$ and the faces of $C$. Recall that, for any positive measurable function $h:(0,\infty)\to(0,\infty)$, $\limsup_{t\da 0}|X_t|/h(t)$ is a.s. a constant on $[0,\infty]$ by Blumenthal's 0--1 law~\cite[Prop.~40.4]{MR1113220}.

\begin{lemma}[Khintchine]
\label{lem:Khintchine}
Suppose $X$ is not compound Poisson with drift. Let $h:(0,\infty)\to(0,\infty)$ be measurable and increasing at $0$ and fix $R>0$. The following statements hold.
\begin{enumerate}
\item[{\nf(i)}] If $\int_0^1  \p(|X_t|/h(t)>R/4)t^{-1}\D t<\infty$, then
$\limsup_{t\downarrow0}|X_t|/h(t)\le R$ a.s.
\item[{\nf(ii)}] If $\int_0^1  \p(|X_t|/h(t)>8R)t^{-1}\D t=\infty$, then $\limsup_{t\downarrow0}|X_t|/h(t)\ge R$ a.s.
\end{enumerate}
\end{lemma}

\begin{remark}
For completeness and accessibility, we give a short elementary proof of Lemma~\ref{lem:Khintchine} in Appendix~\ref{app:khintchine} below. It is based on Khintchine's proof of a closely related result in the Russian text~\cite[Fundamental lemma]{MR0002054}. It is not essential for the results in this paper, but it is natural to enquire whether Lemma~\ref{lem:Khintchine} holds with the constants $R/4$ and $8R$ in the integral conditions substituted by $R$.
\end{remark}

\begin{corollary}
\label{cor:Khintchine}
Suppose $X$ is not compound Poisson with drift. Let $h:(0,\infty)\to(0,\infty)$ be measurable and increasing at $0$. Define the set of $h$-slopes $\mS_h\coloneqq\{\xi_n/h(\ell_n):n\in\N\}$ and set $k_h\coloneqq \sup_{s\in\mS_h}|s|$. Then~$\p(k_h=\infty)\in\{0,1\}$. Moreover, $k_h<\infty$ a.s. if and only if $\limsup_{t\da 0}|X_t|/h(t)<\infty$ a.s.
\end{corollary}

\begin{proof}
Suppose there exists $R\in(0,\infty)$ such that $\int_0^1\p(|X_t|/h(t)>R)t^{-1}\D t<\infty$. Then~\cite[Cor.~3.2]{SmoothCM} (applied to $f(t,x)=|x|/h(t)$) implies that $\mS_h\cap(\R\setminus[-R,R])$ is a.s. a finite set and hence $k_h<\infty$ a.s. Similarly, since $\limsup_{t\da 0}|X_t|/h(t)$ is a.s. constant, Lemma~\ref{lem:Khintchine}(i) implies $\limsup_{t\da 0}|X_t|/h(t)\le 4R$. 

Next assume that for all $R\in(0,\infty)$ we have $\int_0^1\p(|X_t|/h(t)>R)t^{-1}\D t=\infty$. Then~\cite[Cor.~3.2]{SmoothCM} (applied to $f(t,x)=|x|/h(t)$) implies that $\mS_h\cap(\R\setminus[-R,R])$ is a.s. an infinite set for any $R>0$. Hence $k_h\ge R$ a.s. for any $R>0$, implying that $k_h=\infty$ a.s. Similarly, Lemma~\ref{lem:Khintchine}(ii) implies $\limsup_{t\da 0}|X_t|/h(t)\ge R/8$ a.s. for any $R>0$ and hence $\limsup_{t\da 0}|X_t|/h(t)=\infty$.

Since $\int_0^1\p(|X_t|/h(t)>R)t^{-1}\D t$ is either finite for some $R$ or infinite for all $R$, it follows that $\p(k_h=\infty)$ is either $0$ or $1$, respectively. Moreover, the former (resp. latter) case implies that $\limsup_{t\da 0}|X_t|/h(t)$ is finite (resp. infinite) a.s.
\end{proof}

\begin{proof}[Proof of Proposition~\ref{prop:not-1/b}]
First note that, for any $p>0$ the sum $\sum_{n\in\N}\ell_n^{p}$ is finite a.s. (with mean $T^p/p$) by~\cite[Thm~3.1]{SmoothCM}. Pick any $r'>r$ and note that $|\xi_n|/\ell_n^r\le k_{r'}\ell_n^{r'-r}$ for every $n\in\N$, implying 
\begin{equation}
\label{eq:r_r'}
K_r^{1/(1-r)}
=\sum_{s\in\mS_r}|s|^{1/(1-r)}
=\sum_{n\in\N}(|\xi_n|/\ell_n^r)^{1/(1-r)}
\le k_{r'}^{1/(1-r)}\sum_{n\in\N}\ell_n^{(r'-r)/(1-r)}.
\end{equation}
In particular, $K_r<\infty$ whenever $k_{r'}<\infty$ for some $r'>r$.

Assume first that $\sigma^2>0$, then~\cite[Prop.~47.11]{MR3185174} yields $\limsup_{t\da 0}|X_t|/\sqrt{t\log\log(1/t)}=\sqrt{2}|\sigma|>0$. Thus, the limit $\lambda_{1/r}=\limsup_{t\da 0}|X_t|/t^r$ equals $0$ (resp. $\infty$) a.s. for $r\in(0,1/2)$ (resp. $r\in[1/2,1)$). Then, by Corollary~\ref{cor:Khintchine}, we have $k_r=\infty$ for all $r\in[1/2,1)$ and $k_{(r+1/2)/2}<\infty$ for $r\in(0,1/2)$ since $(r+1/2)/2<1/2$. In the latter case, $r<(r+1/2)/2$ and hence $K_r<\infty$ by~\eqref{eq:r_r'}. 

Next assume $\sigma^2=0$. By~\cite[Prop.~47.24]{MR3185174}, $\lambda_{1/r}$ equals $0$ (resp. $\infty$) a.s. for $r\in(0,1/\beta)$ (resp. $r\in(1/\beta,1)$) where $\beta$ is the Blumenthal--Getoor index defined in~\eqref{eq:BG}. As before, by Corollary~\ref{cor:Khintchine}, we have $k_r=\infty$ for all $r\in(1/\beta,1)$ and $k_{(r+1/\beta)/2}<\infty$ for $r\in(0,1/\beta)$ since $(r+1/\beta)/2<1/\beta$. In the latter case, $r<(r+1/\beta)/2$ and hence $K_r<\infty$ by~\eqref{eq:r_r'}. 
\end{proof}

Recall that, by Fubini's theorem,
$J_p=\int_{(-1,1)}|x|^p\nu(\D x)=\int_0^1\ov\nu(t^{1/p})\D t-\ov\nu(1)$ for $p>0$.

\begin{proof}[Proof of Proposition~\ref{prop:1/b}]
Let $r=1/\beta\in[1/2,1)$. By~\cite[Thm~2.1]{MR2370602}, $\int_0^1\ov\nu(t^r)\D t$ is finite (resp. infinite) if and only if $\lambda_{1/r}=\limsup_{t\da 0}|X_t|/t^r$ is finite (resp. infinite) a.s. Thus, by Corollary~\ref{cor:Khintchine}, $\int_0^1\ov\nu(t^r)\D t=\infty$ if and only if $k_r=\infty$. By Proposition~\ref{prop:SB-sum} (with $\phi(x,t)=(|x|/t^r)^{1/(1-r)}$): $I_\beta=\int_0^1\E[\min\{|X_t|/t^r,1\}^{1/(1-r)}]t^{-1}\D t$ is finite if and only if $K_r^{1/(1-r)}=\sum_{n\in\N}|\xi_n|^{1/(1-r)}/\ell_n^{r/(1-r)}$ is finite a.s., completing the proof.
\end{proof}

The following elementary result is required to establish Proposition~\ref{prop:suff_1/b}. 

\begin{lemma}
\label{lem:no-gap}
Let $f:(0,1)\to[0,\infty)$ be a non-increasing function. Then the condition $\int_0^1f(x)\D x<\infty$ implies $\int_0^1x^{p-1}f(x)^p\D x<\infty$ for any $p\ge 1$.
\end{lemma}

\begin{proof}
Observe that, for all $x\in(0,1)$, we have
\begin{align*}
xf(x)
\le g(x)
&\coloneqq\sum_{n\in\N}
    2^{1-n}f(2^{-n})\1_{[2^{-n},2^{1-n})}(x)\\
&=4\sum_{n\in\N}
    (2^{-n}/2)f(2^{1-n}/2)\1_{[2^{-n},2^{1-n})}(x)
\le 4\tfrac{x}{2}f\big(\tfrac{x}{2}\big).
\end{align*}
Thus, defining $w_n\coloneqq 2^{1-n}f(2^{-n})$ for $n\in\N$, we have 
\[
\sum_{n\in\N}w_n\log 2
=\int_0^1 g(x) \frac{\D x}{x}
\le 4\int_0^1
    \tfrac{x}{2}f\big(\tfrac{x}{2}\big)
    \frac{\D x}{x}
= 4\int_0^{1/2}f(x)\D x
<\infty.
\]
In particular, $W\coloneqq\sup_{n\in\N}w_n<\infty$ and hence,
\[
\int_0^1 x^{p-1}f(x)^p\D x
\le\int_0^1 g(x)^p\frac{\D x}{x}
=\sum_{n\in\N}w_n^p\log 2
\le\sum_{n\in\N}w_nW^{p-1}\log 2<\infty.\qedhere
\]
\end{proof}

\begin{proof}[Proof of Proposition~\ref{prop:suff_1/b}]
Recall that $I_\beta<\infty$ implies $J_\beta<\infty$ by Proposition~\ref{prop:1/b}, so it suffices to prove the converse. Define
\[
\varpi(u)\coloneqq\int_{(-1,1)\setminus(-u,u)}x\nu(\D x)\in\R
\quad\text{for $u\in(0,1]$.}
\]
We will show that, under our assumptions, the following integrals are finite:
\begin{gather*}
\mathrm{(i)}\,\int_0^1 t\ov\nu\big(t^{1/\beta}\big)^2\D t<\infty,
\enskip
\mathrm{(ii)}\,\int_0^1 t^{1-2/\beta}\varpi\big(t^{1/\beta}\big)^2\D t<\infty,\\
\mathrm{(iii)}\,\int_0^1\E\big[\min\{|X_t|/t^{1/\beta},1\}^2\big]\frac{\D t}{t}<\infty,
\enskip
{\mathrm{(iv)}}\,\,I_\beta=\int_0^1\E\big[\min\{|X_t|/t^{1/\beta},1\}^{\beta/(\beta-1)}\big]\frac{\D t}{t}<\infty.
\end{gather*}

Let $r=1/\beta$ and note that $\int_0^1\ov\nu(t^r)\D t<\infty$ by assumption. Thus, Lemma~\ref{lem:no-gap} (with $f(x)=\ov\nu(x^r)$) gives (i) $\int_0^1t\ov\nu(t^r)^2\D t<\infty$. Since $1/(1-r)=\beta/(\beta-1)\ge 2$ and $\min\{|x|,1\}^p\le\min\{|x|,1\}^q$ for $p\ge q$, (iii) implies (iv). It remains to show that (ii) and (iii) hold.

Let us establish (ii) $\int_0^1t^{1-2r}\varpi(t^r)^2\D t<\infty$. Denote $\ov\nu_1(x)\coloneqq\ov\nu(x)-\ov\nu(1)$ for $x\in(0,1]$. By Fubini's theorem, we have 
\begin{equation*}
|\varpi(u)|
\le\int_{(-1,1)}\1_{\{u\le |x|<1\}}\int_0^{|x|}\D y\nu(\D x)
=\int_0^1\ov\nu_1(\max\{y,u\})\D y
=u\ov\nu_1(u) + \int_u^1\ov\nu_1(y)\D y.
\end{equation*}
Hence, the elementary inequality $(a+b)^2\le 2(a^2+b^2)$ yields
\[
\frac{1}{2}t^{1-2r}\varpi(t^r)^2
\le t\ov\nu(t^r)^2
    + t^{1-2r}\bigg(\int_{t^r}^1\ov\nu(y)\D y\bigg)^2,
\quad t\in(0,1].
\]
Since (i) $\int_0^1t\ov\nu(t^r)^2\D t<\infty$, to establish (ii) we need only show that $\int_0^1t^{1-2r}(\int_{t^r}^1\ov\nu(y)\D y)^2\D t$ is finite. Since $\min\{a,b\}^2\le ab$ and $r=1/\beta<1$, Fubini's theorem gives
\begin{align*}
2(1-r)\int_0^1t^{1-2r}\bigg(\int_{t^r}^1\ov\nu(y)\D y\bigg)^2\D t
& =2(1-r)\int_0^1\int_0^1\bigg(\int_0^{\min\{x,y\}^{1/r}}t^{1-2r}\D t \bigg)\ov\nu(x)\ov\nu(y)\D x\D y\\
& =\int_0^1\int_0^1 \min\{x,y\}^{2/r-2} \ov\nu(x)\ov\nu(y)\D x\D y\\
& \le\bigg(\int_0^1x^{1/r-1}\ov\nu(x)\D x\bigg)^2
=\bigg(r\int_0^1\ov\nu(t^r)\D t\bigg)^2<\infty.
\end{align*}

It remains to establish (iii) $\int_0^1\E[\min\{|X_t|/t^r,1\}^2]t^{-1}\D t<\infty$.
Let $\gamma$ be the drift parameter of $X$ corresponding to the cutoff function $x\mapsto\1_{\{|x|<1\}}$ (see~\cite[Def.~8.2]{MR3185174}) and recall $\ov\sigma^2(u) 
=\int_{(-u,u)}x^2\nu(\D x)$ for $u>0$.  Apply~\cite[Lem.~A.1]{HowSmoothCM} (with $\ve=K=t^r$) to obtain \[
\int_0^1\E\big[\min\{|X_{t}|,t^r\}^{2}\big]\frac{\D t}{t^{1+2r}}\le \int_0^1 \big[t^2(\gamma-\varpi(t^r))^2  + t\ov\sigma^2(t^r) + t^{2r+1}\ov\nu(t^r)\big]\frac{\D t}{t^{1+2r}}. 
\]
Since the integrals $\int_0^1\gamma^2t^{1-2r}\D t=\gamma^2/(2-2r)$, $\int_0^1t^{1-2r}\varpi(t^r)^2\D t$ and $\int_0^1\ov\nu(t^r)\D t$ are all finite (recall $r=1/\beta<1$), it remains to show that $\int_0^1t^{-2r}\ov\sigma^2(t^r)\D t<\infty$. 

By Fubini's theorem, we obtain
\[
\ov\sigma^2(x)
=\int_{(-x,x)} \bigg(\int_0^{|u|} 2y\D y\bigg)\nu(\D u)
=2\int_0^x y(\ov\nu(y)-\ov\nu(x))\D y
\le 2\int_0^x y\ov\nu(y)\D y,
\quad x\in(0,1].
\]
Consider the case $r=1/\beta\in(1/2,1)$. Fubini's theorem gives
\[
\int_0^1t^{-2r}\ov\sigma^2(t^r)\D t
\le 2\int_0^1\int_0^{t^r} t^{-2r}y\ov\nu(y)\D y\D t
= 2\int_0^1\frac{y^{1/r-2}-1}{2r-1}y\ov\nu(y)\D y,
\]
which is finite since $\int_0^1y^{1/r-1}\ov\nu(y)\D y=r\int_0^1\ov\nu(t^r)\D t<\infty$. Now consider the case $r=1/\beta=1/2$. Again by Fubini's theorem, we obtain
\begin{align*}
\int_0^1t^{-1}\ov\sigma^2(\sqrt{t})\D t
&\le 2\int_0^1\int_0^{\sqrt{t}} t^{-1}y\ov\nu(y)\D y\D t\\
&= 4\int_0^1\log(1/y)y\ov\nu(y)\D y
= \int_0^1\log(1/x)\ov\nu(\sqrt{x})\D x<\infty.\qedhere
\end{align*}
\end{proof}

\begin{proof}[Proof of Proposition~\ref{prop:b=2}]
Recall that $\lambda_2=\limsup_{t\da 0}|X_t|/\sqrt{t}\in[0,\infty]$ by~\cite[Thm~2.1]{MR2370602}. If $\lambda_2=\infty$, then $k_{1/2}=\infty$ by Corollary~\ref{cor:Khintchine}. If $\lambda_2\in[0,\infty)$ then $k_{1/2}<\infty$ by Corollary~\ref{cor:Khintchine}. Finally, assume $\lambda_2\in(0,\infty)$. Then $\int_0^1t^{-1}\p(|X_t|/\sqrt{t}>R)\D t=\infty$ 
for $R<\lambda_2/4$ by Lemma~\ref{lem:Khintchine}(i). Thus, for any $\ve\in(0,\lambda_2/4)$, $\mS_{1/2}$ has infinitely many points with magnitude on the interval $[\ve,\infty)$ by~\cite[Cor.~3.2]{SmoothCM}, 
implying $K_{1/2}=\infty$.
\end{proof}

\section{Concluding remarks}
\label{sec:conclusion}

It is natural to consider the question of whether the convex minorant $C$ is $h$-H\"older continuous, i.e., if $\sup_{0\le u<t\le T}|C_t-C_u|/h(t-u)<\infty$, for an appropriate general concave increasing function $h:(0,\infty)\to(0,\infty)$. In this context, it is also easy to see that 
\[
\sup_{0\le u<t\le T}\frac{|C_t-C_u|}{h(t-u)}
\ge k_h=\sup_{n\in\N}\frac{|\xi_n|}{h(\ell_n)}
=\sup_{s\in\mS_h}|s|,
\] 
where the finiteness of $k_h$ can be completely characterised via Corollary~\ref{cor:Khintchine} and part~(a) of the main theorem in~\cite{MR968135} in terms of the L\'evy measure $\nu$ (see  Corollary~\ref{cor:Kim_Wee_Khintchine} in Subsection~\ref{subsec:k_h} below for details). 

It is not, however, immediately clear how to construct a tractable upper bound, say $K_h$, satisfying $K_h<\infty$ whenever $k_h<\infty$. Indeed, a crucial step in proving Lemma~\ref{lem:PC-Holder} (and hence~\eqref{eq:Holder-K_r}) is the application of H\"older's inequality to establish that the $r$-H\"older constant of $C$ is bounded by the $L^p$-norm of the derivative $C'$ for $p=1/(1-r)$. This step is not easily extendable to a general concave function $h$ since there is no sufficiently sharp extension of H\"older's inequality (see, e.g.~\cite{MR1436394,MR2489348}). Thus, it appears that a generalisation of our results beyond the case where $h$ is a power function would require analysing the integral $\int_u^t|C'_v|\D v$ for all $0\le u<t\le T$ by other means. For instance, the results in~\cite{HowSmoothCM} obtain upper and lower functions for $|C'|$ at $0$ and $T$, yielding upper and lower bounds on $\int_u^t|C'_v|\D v$ for $u<t$ close to either $0$ or $T$. Note however, that there may exist a large gap between the upper and lower functions of $C'$, see~\cite[Rem.~2.14(a)]{HowSmoothCM}, showing that this question is nontrivial. 


\subsection{When is \texorpdfstring{$k_h$}{kh} finite?}
\label{subsec:k_h}

The following corollary is a direct consequence of Corollary~\ref{cor:Khintchine} and part (a) of the main theorem in~\cite{MR968135}. Recall the definition of $\ov\nu$ and $\varpi$ in~\eqref{eq:ov_nu} and let $\gamma$ be the drift parameter of $X$ (for the cutoff function $x\mapsto\1_{(-1,1)}(x)$, see~\cite[Def.~8.2]{MR3185174}) and $\ov\sigma^2(u)=\int_{(-u,u)}x^2\nu(\D x)$ for $u>0$.

\begin{corollary}
\label{cor:Kim_Wee_Khintchine}
Suppose $X$ is not compound Poisson with drift. Then, for any function $h$ increasing at $0$ with $h(0)=0$, the variable $k_h<\infty$ a.s. (resp. $k_h=\infty$ a.s.) if and only if $\limsup_{t\downarrow0}|X_t|/h(t)$ is a.s. finite (resp. infinite). Moreover, the following statements hold.
\begin{enumerate}
\item[{\nf{(i)}}] If $\sigma^2>0$, then $k_h<\infty\text{ a.s.}$ if and only if $\liminf_{t\da 0}h(t)/\sqrt{t\log\log(1/t)}>0\text{ a.s.}$
\item[{\nf{(ii)}}] If $\sigma^2=0$ and $\limsup_{x \da 0} (x^{-2}\ov\sigma^2(x)+x^{-1}|\gamma-\varpi(x)|)/\ov \nu(x)<\infty$, then the random variable $k_h<\infty\text{ a.s.}$ if and only if $\int_0^1 \ov\nu(h(t))\D t<\infty$. 
\item[{\nf{(iii)}}] Suppose $\sigma^2=0$ and $\limsup_{x \da 0} (x^{-2}\ov\sigma^2(x)+x^{-1}|\gamma-\varpi(x)|)/\ov \nu(x)=\infty$. Then there exists a non-decreasing function $h^*$ such that $\limsup_{t\da0}|X_t|/h^*(t)\in(0,\infty)$ a.s. ($h^*$ constructed in the paragraph below). Moreover, the following implications hold
\begin{gather*}
\limsup_{t\da0} h^*(t)/h(t)<\infty
\enskip\implies\enskip
\text{$k_h<\infty$ a.s.}\\
\liminf_{t\da0} h^*(t)/h(t)=\infty
\enskip\implies\enskip
\text{$k_h=\infty$ a.s.}
\end{gather*}
\end{enumerate}
\end{corollary}

Wee and Kim~\cite{MR968135} proved that  $\limsup_{t \da 0}|X_t|/h^*(t)\in (0,\infty)$ a.s. for a non-decreasing function $h^*$ if and only if $\sigma^2=0$ and $\liminf_{x \da 0} \ov \nu(x)/(\ov\nu(x)+x^{-2}\ov\sigma^2(x)+x^{-1}|\gamma-\varpi(x)|)=0$. In the following two cases, which are exhaustive by~\cite[Lem.~3.3]{MR968135}, we describe a construction of the function $h^*$, implicitly given in the proof of~\cite[Thm~3.4]{MR968135}. 

\begin{enumerate}
    \item[\nf{(a)}] Suppose $\liminf_{x \da 0}(\ov\nu(x)+x^{-1}|\gamma-\varpi(x)|)/(x^{-2}\ov\sigma^2(x))=0$. Choose a sequence $u_n\da 0$, such that $u_{n+1}^{-2}\ov\sigma^2(u_{n+1})>2u_n^{-2}\ov\sigma^2(u_n)$ for $n\in\N$ and $\sum_{n \in \N}\log(n)(\ov\nu(u_n)+u_n^{-1}|\gamma-\varpi(u_n)|)/(u_n^{-2}\ov\sigma^2(u_n))<\infty$. Let $t_n=\log(n)/(u_n^{-2}\ov\sigma^2(u_n))$ for all $n \in \N$, and define $h^*(t)\coloneqq u_n\log(n)$ for $t_{n+1}<t\le t_n$ and $n \in \N$.
    \item[\nf{(b)}] Suppose $\liminf_{x \da 0} (\ov\nu(x)+x^{-2}\ov\sigma^2(x))/(x^{-1}|\gamma-\varpi(x)|)=0$. Choose a sequence $u_n\da 0$, such that $u_{n+1}^{-1}|\gamma-\varpi(u_{n+1})|\ge 2u_n^{-1}|\gamma-\varpi(u_n)|$ for $n \in \N$ and  $\sum_{n \in \N}(\ov\nu(u_n)+u_n^{-2}\ov\sigma^2(u_n))/(u_n^{-1}|\gamma-\varpi(u_n)|)<\infty$. Let $t_n=1/(u_n^{-1}|\gamma-\varpi(u_n)|)$ and define $h^*(t)\coloneqq u_n$ for $t_{n+1}<t\le t_n$ and $n \in \N$.
\end{enumerate}

\appendix
\section{Proof of Lemma~\ref{lem:Khintchine}}
\label{app:khintchine}

We present a short proof of Lemma~\ref{lem:Khintchine}, based on the proof of~\cite[Fundamental lemma]{MR0002054}. 

\begin{proof}[Proof of Lemma~\ref{lem:Khintchine}]
Fix $0<s<t$ and $0<y<x$, then  $\{|X_t|\ge x\}\subset\{|X_s|\ge y\}\cup\{|X_t-X_s|\ge x-y\}$. Since $X_t-X_s\eqd X_{t-s}$, this yields 
\begin{equation}\label{eq:Kh_mid} 
\p(|X_t|\ge x)
\le\p(|X_s|\ge y)+\p(|X_{t-s}|\ge x-y).
\end{equation}
In particular, taking $s=t/2$ and $y=x/2$ gives
$P(t,x)\coloneqq\p(|X_t|\ge x)\le 2P(t/2,x/2)$. Without loss of generality, we assume throughout that $h$ is non-decreasing on $(0,1]$.

\emph{Part (i).} 
It suffices to show that, given $R>0$, the condition $\int_0^1 P(t,Rh(t))t^{-1}\D t<\infty$ implies $\limsup_{t\da 0}|X_t|/h(t)\le 4R$ a.s. The proof is split in three steps.

\underline{Step 1.} We first show that $P(t,2Rh(t))\to 0$ as $t\da 0$ and, in particular, there exists some $\ve>0$ such that $P(t,2Rh(t))<1/2$ for $t\in(0,\ve)$. Denote $p(t)=P(t,Rh(t))$ for $t>0$. Since $h$ is non-decreasing, \eqref{eq:Kh_mid} implies $P(t,2Rh(t))\le p(s)+p(t-s)$. Integrating the previous inequality over $[t/2,t]$ with respect to the measure $s^{-1}\D s$ yields 
\begin{align*}
P(t,2Rh(t))\log 2
& \le \int_{t/2}^t p(s)\frac{\D s}{s}
+\int_{t/2}^t p(t-s)\frac{\D s}{s}\\
& \le \int_{t/2}^t p(s)\frac{\D s}{s}
+\int_{t/2}^t p(t-s)\frac{\D s}{t-s}
= \int_0^t p(s)\frac{\D s}{s}
<\infty.
\end{align*}
Thus, the limit $\lim_{t\da 0}\int_0^t p(s)s^{-1}\D s=0$ implies  $\lim_{t\da 0}P(t,2Rh(t))=0$. 

\underline{Step 2.} Define $\ov{X}_t\coloneqq \sup_{s\in[0,t]}X_s$ for $t\ge 0$. We will show that $\p(\ov X_t> 4Rh(t))\le 2\p(X_t> 2Rh(t))$ for $t\in(0,\ve)$ where $\ve$ is as in Step 1. Fix $n\in\N$, set $t_{k}\coloneqq tk/n$ for $k\in\{1,\ldots,n\}$ and define the events
\[
A_k:=\{X_{t_i}\le 4Rh(t)\text{ for all }i\in\{1,\ldots,k-1\}\}\cap\{X_{t_k}> 4Rh(t)\},
\quad k\in\{1,\ldots,n\}.
\]
Since the increments of $X$ are independent and stationary, we have 
\begin{align*}
\p(X_t> 2Rh(t)|A_k)
&\ge \p(X_t-X_{t_k}>-2Rh(t)|A_k)
= \p(X_t-X_{t_k}>-2Rh(t))\\
&\ge \p(|X_t-X_{t_k}|< 2Rh(t))
= \p(|X_{t-t_k}|< 2Rh(t))\\
& \ge \p(|X_{t-t_k}|< 2Rh(t-t_k)).
\end{align*}
By step 1, $t-t_k<\ve$ for $t\in(0,\ve)$ and hence $\p(X_t> 2Rh(t)|A_k)>1/2$ for $k\in\{1,\ldots,n\}$. 

Define $M^{(n)}_t\coloneqq\max_{1\le k\le n}X_{t_k}$, then $\{M_t^{(n)}> 4Rh(t)\}=\bigcup_{k=1}^n A_k$. Since the sets $A_k$ are disjoint, for any
$t\in(0,\ve)$ we have 
\[
\p\big(M_t^{(n)}> 4Rh(t)\big)
=\sum_{k=1}^n\p(A_k)
\le 2\sum_{k=1}^n \p(A_k)\p(X_t> 2Rh(t)|A_k)
\le 2\p(X_t> 2Rh(t)).
\]
Since $X$ is~\cadlag, $M^{(2^n)}_t\ua \ov X_t$ a.s. as $n\to\infty$. Hence, the monotone convergence theorem yields $\p(\ov X_t> 4Rh(t))=\lim_{n\to\infty}\p(M^{(2^n)}_t> 4Rh(t))\le 2\p(X_t> 2Rh(t))$. 

\underline{Step 3.} Define $p_n 
\coloneqq \p(\sup_{t\in[2^{-n}, 2^{1-n}]}(X_t/h(t))>4R)$ for $n\in\N$ and let $n_\ve$ be the smallest positive integer larger than $1+\log(1/\ve)/\log 2$, where $\ve$ is as in Step 1. Since $h$ is non-decreasing, Step 2 and~\eqref{eq:Kh_mid} imply that for all $n\ge n_\ve$ and $t\in[2^{-n},2^{1-n}]$ we have 
\begin{multline*}
p_n
\le \p\bigg(\sup_{t\in[2^{-n}, 2^{1-n}]}X_t> 4R h(2^{-n})\bigg)
= \p\big(\ov X_{2^{-n}}> 4R h(2^{-n})\big)
\le \p\big(\ov X_{t}> 4R h(2^{-n})\big)\\
\le 2\p(X_{t}> 2R h(2^{-n}))
\le 2\p(X_{t}> 2R h(t/2))
\le 4\p(X_{t/2}> R h(t/2)).
\end{multline*}
Integrating the previous inequality over $t\in[2^{-n},2^{1-n}]$ and summing over $n\ge n_\ve$ gives 
\begin{align*}
\sum_{n=n_\ve}^\infty p_n \frac{\log 2}{4} &
=\sum_{n=n_\ve}^\infty \int_{2^{-n}}^{2^{1-n}}\frac{p_n}{4} \frac{\D t}{t}
\le\sum_{n=1}^\infty \int_{2^{-n}}^{2^{1-n}}\p(X_{t/2}> R h(t/2)) \frac{\D t}{t}\\
& = \int_{0}^{2}\p(X_{t}> R h(t)) \frac{\D t}{t}<\infty.
\end{align*}
The Borel--Cantelli lemma implies  $\sup_{t\in[2^{-m-1},2^{-m}]}(X_t/h(t))\le 4R$ for all but finitely many $n$, implying $\limsup_{t\to 0}X_t/h(t)\le 4R$ a.s. By symmetry, $\limsup_{t\to 0}(-X_t)/h(t)\le 4R$ a.s., proving part (i). 

\emph{Part (ii).}
It suffices to show that, given $R>0$, if $\int_0^1P(t,8Rh(t))t^{-1}\D t=\infty$ (recall $P(t,x)=\p(|X_t|>x)$) then $\limsup_{t\da 0}|X_t|/h(t)\ge R$ a.s. The proof requires three steps.

\underline{Step 1.} Define $M(t)\coloneqq\sup_{s\in(0,t]}(|X_s|/h(s))$. We will show that 
\begin{equation}\label{eq11}
B_n\coloneqq\bigg\{
\sup_{t\in[2^{-n-1},2^{-n}]}
    |X_t-X_{2^{-n-1}}| > 2Rh(2^{-n}),\, M(2^{-n-1})\leq R\bigg\}
\subset\{M(2^{-n})> R\}.
\end{equation}
To see~\eqref{eq11} note that, on the event $B_n$ there exists some $t\in [2^{-n-1},2^{-n}]$ satisfying
\[
M(2^{-n})
\ge |X_t| 
\ge |X_t-X_{2^{-n-1}}|-|X_{2^{-n-1}}|
>2Rh(2^{-n})-Rh(2^{-n-1})
\ge Rh(2^{-n})
\ge Rh(t).
\]

\underline{Step 2.} We claim $\sum_{n\in\N}q_n=\infty$, where $q_n\coloneqq \p(\sup_{t\in[0,2^{-n-1}]}|X_t|>2Rh(2^{-n}))$. For $t \le 2^{-n-1}$, apply~\eqref{eq:Kh_mid} twice to get $4q_n
\ge 4P(t,2Rh(2^{-n}))
\ge P(4t,8Rh(2^{-n}))$. Hence, for any $t \in [2^{-n-2},2^{-n-1}]$, we have $4q_n \ge P(4t,8Rh(4t))$.
Integrating the previous inequality on $[2^{-n-2},2^{-n-1}]$ with respect to $t^{-1}\D t$ yields
\begin{equation*}
(4\log 2)q_n  
\ge \int_{2^{-n-2}}^{2^{-n-1}} P(4t,8Rh(4t))\frac{\D t}{t} 
=\int_{2^{-n}}^{2^{-n+1}} P(t,8Rh(t)) \frac{\D t}{t}, 
\quad \text{for all }n \in \N.
\end{equation*}
Thus, the condition $\int_0^1 P(t,8Rh(4t))t^{-1}\D t=\infty$ implies $\sum_{n\in \N} q_n=\infty$.

\underline{Step 3.} Define $r_n:=\p(M(2^{-n})>R)$ for $n\in \N\cup \{0\}$. By Step 1, the event $B_n\subset\{M(2^{-n-1})\le R\}$ in~\eqref{eq11} satisfies
\[
q_n(1-r_{n+1})=\p(B_n)
\le\p(M(2^{-n-1})\le R,\, M(2^{-n})>R)
=r_n-r_{n+1}.
\]
This further implies that, for any $k\ge0$ and $n\in\N$,
\begin{equation*}
0\le 1-r_n
\le (1-r_{n+1})(1-q_n) 
\le (1-r_{n+k+1})\prod_{i=n}^{n+k}(1-q_i).
\end{equation*}
Since $\sum_{n\in \N} q_n=\infty$, it follows that $\prod_{i=n}^{n+k}(1-q_i) \to 0$ as $k \to \infty$. Indeed, if $q_\infty\coloneqq\limsup_{k\to\infty}q_k>0$, then the limit is obvious since $1-q_i\le 1-q_\infty/2<1$ for infinitely many $i\ge n$ and, if $q_\infty=0$, the result follows from~\cite[Lem.~5.8]{MR1876169}. Hence, we have $r_n=1$ for all $n \in \N\cup\{0\}$ and thus $\limsup_{t \da 0} |X_t|/h(t)=\lim_{t\da 0}M(t)\ge R$ a.s. 
\end{proof}

\printbibliography

\section*{Acknowledgements}

\thanks{
\noindent JGC and AM are supported by the EPSRC grant EP/V009478/1 and The Alan Turing Institute under the EPSRC grant EP/N510129/1; 
AM was also supported by the EPSRC grant 
EP/W006227/1 and
the Turing Fellowship funded by the Programme on Data-Centric Engineering of Lloyd's Register Foundation; DKB is funded by the CDT in Mathematics and Statistics at The University of Warwick. All three authors would like to thank the Isaac Newton Institute for Mathematical Sciences in Cambridge, supported by EPSRC grant  EP/R014604/1, for hospitality during the programme on Fractional Differential Equations where part of this work was undertaken. We also want to thank the anonymous referees whose questions led us to complete the characterisation in the case $\beta\in(1,2)$.
}

\end{document}